\documentclass[12pt]{amsart}
\usepackage{amssymb,latexsym}
\usepackage{amsfonts}
\usepackage{amsmath}
\usepackage[colorlinks,linkcolor=blue,anchorcolor=blue,citecolor=blue]{hyperref}
\usepackage{algorithm}
\usepackage{enumerate}
\usepackage{algpseudocode}
\usepackage{verbatim}
\usepackage{graphicx}

\newcommand\C{{\mathbb C}}

\newcommand\Z{{\mathbb Z}}

\newtheorem{theorem}{Theorem}[section]
\newtheorem{lemma}[theorem]{Lemma}

\newtheorem{corollary}[theorem]{Corollary}

\theoremstyle{definition}

\newtheorem{remark}[theorem]{Remark}

\numberwithin{equation}{section}

\newcommand\Res{\mathrm{Res}}

\newcommand\q{\mathfrak{q}}

\begin{document}

\title[Lucas and Lehmer sequences]{On the Lucas and Lehmer sequences in Dedekind domains}

\author{Xiumei Li}
\address{School of Mathematical Sciences, Qufu Normal University, Qufu, 273165, China}
\email{lxiumei2013@qfnu.edu.cn}

\author{Giulio Peruginelli}
\address{Department of Mathematics, University of Padova, Via Trieste, 63, 35121 Padova, Italy}
\email{gperugin@math.unipd.it}

\author{Min Sha}
\address{\parbox{\linewidth}{School of Mathematical Sciences, South China Normal University, Guangzhou, 510631, China \\ 
Key Laboratory of Applied Mathematics (Putian University), Fujian Province University, Fujian Putian, 351100, China}}
\email{min.sha@m.scnu.edu.cn}



\begin{abstract}
In this paper, we first obtain  the strong divisibility property for the Lucas and Lehmer sequences
in Dedekind domains, and then establish analogues of Zsigmondy's theorem and the primitive divisor results for such sequences in function fields.
\end{abstract}

\maketitle

\section{Introduction}

\subsection{Background and motivation}
Given a sequence $(a_n)_{n \ge 1}$ of the rational integers $\Z$,
it is called a \textit{divisibility sequence} if $a_m \mid a_n$ whenever $m \mid n$,
and it is called a \textit{strong divisibility sequence} if $\gcd(a_m, a_n) = a_j$ with $j = \gcd(m,n)$ for any positive integers $m, n$.
In addition, for $n \ge 2$, a prime divisor of the $n$-th term $a_n$ is called \textit{primitive} if it does not divide $a_k$ for any $k$ with $1 \le k <n$.
The history of studying primitive prime divisors started in the 19th century with work of Bang \cite{Bang} and Zsigmondy \cite{Zsig},
which is now called Zsigmondy's theorem and says that
every term beyond the sixth in the sequence
$(a^n - b^n)_{n \ge 1}$ has a primitive prime divisor, where $a,b$ are  positive coprime integers.
This theorem was independently rediscovered by Birkhoff and Vandiver \cite{BV}.

Carmichael \cite{Car} established an analogue of Zsigmondy's theorem for the Lucas sequence $((a^n - b^n)/(a-b))_{n \ge 1}$,
where $a, b$ are real algebraic integers such that $a/b$ is not a root of unity, and $a+b$ and $ab$ are coprime integers in $\Z$.
Ward \cite{Ward} obtained such a result for the Lehmer sequence $(s_n)_{n \ge 1}$ with $s_n = (a^n - b^n)/(a-b)$ for odd $n$
and $s_n = (a^n - b^n)/(a^2-b^2)$ for even $n$, where $a,b$ are real, and $(a+b)^2$ and $ab$ are coprime integers in $\Z$.
All these results, including Zsigmondy's theorem, were extended to any number field
(that is, $a$ and $b$ are relatively prime integers of a number field such that $a/b$ is not a root of unity)
by Schinzel \cite{Sch} 
in an effective but not explicit manner (see \cite{PS} for an earlier work), 
which was first made explicitly by Stewart \cite{Stewart}. 
Furthermore, Voutier \cite[Conjecture 1]{Voutier} conjectured that for any integer $n > 30$, the $n$-th term of a Lucas or Lehmer sequence always has a primitive prime divisor; and then in 2001, Bilu, Hanrot and Voutier \cite{BHV} proved this conjecture and so listed all the Lucas and Lehmer numbers without primitive prime divisor. 
We refer to \cite[Section 6.1]{BR} for more history about this topic. 

Recently, Flatters and Ward \cite{FW} found an analogue of Zsigmondy's theorem for a polynomial sequence $(f^n - g^n)_{n \ge 1}$,
where $f,g$ are two coprime polynomials in a polynomial ring $K[x]$ in one variable $x$ over any arbitrary field $K$.
This was recently extended to the polynomial rings of several variables over any arbitrary field by one of us; see \cite{Sha}.
Moreover, the author in \cite{Sha} obtained the strong divisibility property and the primitive divisor results for the Lucas and Lehmer sequences
in the polynomial rings of several variables. 
Most recently, as a special case of Lucas sequences, the generalized Fibonacci sequences of polynomials over finite fields have been further studied about their periodic property; see \cite{CSW}. 

In this paper, we try to obtain analogues of Zsigmondy's theorem and
establish the strong divisibility property and the primitive divisor results for the Lucas and Lehmer sequences
in Dedekind domains.
Particularly, this includes extensions of the results in \cite{Car,PS,Sch,Ward} from number fields to function fields.

The approach in this paper is similar to that in \cite{Sha}.
However, the setting here is different (see Section~\ref{sec:setup}), and it indeed needs some extra considerations for Dedekind domains and also for the special case of function fields.

\subsection{Set-up}
\label{sec:setup}
Throughout the paper, let $R$ be a Dedekind domain unless specified otherwise. 
That is, $R$ is an integral domain in which every non-zero proper ideal factors into a product of prime ideals. 
We also assume that $R$ is not a field. 
Let $p$ be the characteristic of $R$. Notice that either $p=0$ or $p$ is a prime number.

An element $u \in R$ is called a \textit{unit} of $R$ if $1/u \in R$.
Two elements of $R$ are called \textit{coprime} if they are not both divisible by the same non-unit in $R$.
Two elements $a$ and $b$ in $R$ are called \textit{associated} if $a = u b$ for a unit $u$ of $R$,
which is equivalent to  $\langle a \rangle = \langle b \rangle$.
Here $\langle a \rangle$ means the ideal of $R$ generated by $a$.
In addition, two ideals $I$ and $J$ of $R$ are called \textit{coprime} if $I+J=R$.

As usual,  a \textit{common divisor} of two elements $a$ and $b$ of $R$ is an element of $R$ dividing both $a$ and $b$.
A \textit{greatest common divisor} of $a$ and $b$, denoted by $\gcd(a,b)$,
is defined as a common divisor of $a$ and $b$ which is divisible by any other common divisor.
We remark that a greatest common divisor may not exist and may be not unique.
Moreover, if $a$ and $b$ have more than one greatest common divisors, then all their greatest common divisors are associated.

By the way, for any two positive integers $m$ and $n$ we still use $\gcd(m,n)$ to denote their greatest common divisor,
which is a positive integer.

For any ideal $I$ and any element $b$ of $R$, we say that $I$ divides $b$ (or write $I \mid b$)
if $I$ divides the ideal $\langle b \rangle$ generated by $b$ in $R$, which is equivalent to $b \in I$.
Besides, a \textit{prime divisor} of $b$ means a prime ideal of $R$ dividing the ideal $\langle b \rangle$.

Given a sequence $(a_n)_{n \ge 1}$ of $R$, a prime ideal $\q$ of $R$ is called
a \textit{primitive prime divisor} of a term $a_n$ if $\q \mid a_n$ but $\q$ does not divide any earlier term in the sequence.
Also, the sequence is called a \textit{strong divisibility sequence}
if for any positive integers $m, n$ with $j = \gcd(m,n)$, we have $\gcd(a_m, a_n)=a_j$
(that is, $\gcd(a_m, a_n)$ exists and it can be equal to $a_j$).

We remark that the above concept of primitive prime divisor follows from the number field case; see \cite{PS, Sch}.

Throughout this paper we fix an algebraic closure of the fraction field of $R$. 
For simplicity, we also call it the algebraic closure of $R$. 
Then, the roots of a non-constant polynomial over $R$ in this algebraic closure are called \textit{algebraic elements} over $R$, and the roots of a non-constant monic polynomial over $R$ are called \textit{integral elements} over $R$. In particular, for any integer $n \ge 1$, the roots of the polynomial $x^n -1$ over $R$ are called roots of unity. 

Finally, we recall a typical example of Dedekind domains (See \cite{Cla, GP, GP2} for more examples). 
In the sequel, we fix a field $K$. 
Let $x$ be a variable which is transcendental over $K$ 
(we fix this variable throughout the paper).  
Let $K[x]$ be the polynomial ring in $x$ over $K$. 
We denote the fraction field of $K[x]$ by $K(x)$, which is called the \textit{rational function field} in $x$. 
A finite extension of $K(x)$ is called a \textit{function field} (unless otherwise specified, we always keep this meaning of function field throughout the paper).
It is well-known that the integral closure of $K[x]$ in a function field is a Dedekind domain.
For simplicity, we call it the \textit{integral closure} of the relevant function field.  

Let $R$ be the integral closure of a function field $E$. 
That is, $E$ is a finite extension of $K(x)$, and $R$ consists of all the integral elements of $E$ over $K[x]$. 
We defined the \textit{degree} of $R$, denoted by $\deg R$, to be the degree of the field extension $E/K(x)$, that is, 
$$
\deg R = [E : K(x)].
$$
Note that the fractional field of $R$ is exactly $E$. 
So, the degree of $R$ equals to the degree of the fractional field of $R$ over $K(x)$.  
Let $d=\deg R$. Recall that $p$ is the characteristic of $R$. We introduce the following condition 
on a positive integer $n$, which will be used frequently later on: 
\begin{equation}  \label{eq:boundn}
\left\{\begin{array}{ll}
 \varphi(n)(\varphi(n)-1)(\varphi(n)-2) > 3^{2d-2}  & \textrm{if $p=0$}, \\
\\
p \nmid n 
\textrm{ and }  \varphi(n)(\varphi(n)-1)(\varphi(n)-2) > 31\cdot 19^{2d-2} & \textrm{if $p>0$},
\end{array}\right.
\end{equation} 
where $\varphi$ is Euler's totient function. 
Using some explicit lower bound for $\varphi(n)$, one can get some explicit lower bound for $n$ to ensure \eqref{eq:boundn}. 
For example, using $\varphi(n) \ge \sqrt{n}$ when $n > 2$, we obtain that if  
\begin{equation}   \label{eq:boundn2}
\left\{\begin{array}{ll}
 n > \left(3^{(2d-2)/3} + 2 \right)^2  & \textrm{if $p=0$}, \\
\\
p \nmid n 
\textrm{ and }  n > \left( 31^{1/3}\cdot 19^{(2d-2)/3} + 2 \right)^2 & \textrm{if $p>0$},
\end{array}\right.
\end{equation} 
then \eqref{eq:boundn} holds.

\subsection{Main results}    \label{sec:main}

In the sequel, $R$ is a Dedekind domain unless otherwise specified. 
We first present an analogue of Zsigmondy's theorem in $R$.
Let $f, g$ be non-zero elements in $R$ 
 such that  the quotient $f/g$ is not a root of unity.
Assume that the principal ideals $\langle f \rangle$ and $\langle g \rangle$ in $R$ are coprime ideals which are not both equal to $R$.
Define the sequence $(F_n)_{n \ge 1}$ of $R$ with respect to $f, g$ by
$$
F_n = f^n - g^n, \quad n = 1,2, \ldots.
$$
Since $f/g$ is not a root of unity, $F_n \ne 0$ for any integer $n \ge 1$.

\begin{theorem}  \label{thm:strong1}
The sequence $(F_n)_{n \ge 1}$ is a strong divisibility sequence.
\end{theorem}

\begin{theorem}[Zsigmondy's theorem in $R$]  \label{thm:primitive1}
Assume that $R$ is the integral closure of a function field, and let $d$ be the degree of $R$. 
Then in the sequence $(F_n)_{n \ge 1}$, each term $F_n$ has a primitive prime divisor whenever $n$ satisfies \eqref{eq:boundn}.
\end{theorem}

\begin{remark}
In Theorem~\ref{thm:primitive1}, when the characteristic $p=0$, in view of \eqref{eq:boundn2} 
we know that each term $F_n$ has a primitive prime divisor  whenever
 $n > \left(3^{(2d-2)/3} + 2 \right)^2$.
\end{remark}

For $f, g$ as above, define the sequence:
$$
G_n = f^n + g^n, \quad n = 1,2, \ldots.
$$
If $p \ne 2$, then $(G_n)_{n \ge 1} \ne (F_n)_{n \ge 1}$.
Note that $F_{2n} = F_n G_n$, which implies that each primitive prime divisor of $F_{2n}$ comes from $G_n$.
Then, the following corollary is a direct consequence of Theorem~\ref{thm:primitive1}.

\begin{corollary}  \label{cor:primitive1}
Assume that $R$ is the integral closure of a function field, and let $d$ be the degree of $R$. 
Then in the sequence $(G_n)_{n \ge 1}$, each term $G_n$ has a primitive prime divisor whenever $n$ satisfies \eqref{eq:boundn}.
\end{corollary}

Now, we consider the Lucas sequences.
Let $\alpha, \beta$ be non-zero algebraic elements over $R$ 
 such that the quotient $\alpha/\beta$ is not a root of unity.
Assume that both $\alpha + \beta$ and $\alpha \beta$ are in $R$,
and the principal ideals $\langle \alpha + \beta \rangle$ and $\langle \alpha \beta \rangle$ in $R$ are coprime ideals
which are not both equal to $R$.
Define the \textit{Lucas sequence} $(L_n)_{n \ge 1}$ of $R$ with respect to $\alpha, \beta$ by
$$
L_n = \frac{\alpha^n - \beta^n}{\alpha - \beta}, \quad n =1,2, \ldots.
$$
We remark that the Lucas sequence $(L_n)_{n \ge 1}$ satisfies the following recurrence relation over $R$:
$$
L_{n+2} = (\alpha + \beta) L_{n+1} - \alpha\beta L_n, \quad n = 1, 2, \ldots.
$$

We also remark that classically the index $n$ of the Lucas sequences starts from $0$. 
Here for simplicity and for consistency we let it start from $1$, and 
we do the same thing for the Lehmer sequences below.

\begin{theorem}  \label{thm:strong2}
The sequence $(L_n)_{n \ge 1}$ is a strong divisibility sequence.
\end{theorem}

\begin{theorem}  \label{thm:primitive2}
Assume that $R$ is the integral closure of a function field containing $\alpha$ and $\beta$, and let $d$ be the degree of $R$. 
Then in the sequence $(L_n)_{n \ge 1}$, each term $L_n$ has a primitive prime divisor whenever $n$ satisfies \eqref{eq:boundn}.
\end{theorem}

\begin{remark}
If $R$ does not contain $\alpha$ and $\beta$, then 
 since $\alpha$ and $\beta$ satisfy the equation $x^2 - (\alpha + \beta)x + \alpha \beta = 0$ over $R$, 
we can get an extension of $R$ whose degree is at most $2d$ and which contains $\alpha$ and $\beta$, 
and so the result in Theorem~\ref{thm:primitive2} still holds if we replace $d$ by $2d$ in \eqref{eq:boundn}. 
\end{remark}

Now, we consider the Lehmer sequences.
Let $\lambda, \eta$ be non-zero algebraic elements over $R$ 
 such that $\lambda / \eta$ is not a root of unity.
Assume that both $(\lambda + \eta)^2$ and $\lambda\eta$ are in $R$,
and  the principal ideals $\langle (\lambda + \eta)^2 \rangle$ and $\langle \lambda\eta \rangle$ of $R$ are  coprime ideals
 which are not both equal to $R$.
Define the \textit{Lehmer sequence} $(U_n)_{n \ge 1}$ of $R$ with respect to $\lambda, \eta$ by
\begin{equation*}
U_n =
\left\{\begin{array}{ll}
\dfrac{\lambda^n - \eta^n}{\lambda - \eta} & \textrm{if $n$ is odd,}\\
\\
\dfrac{\lambda^n - \eta^n}{\lambda^2 - \eta^2} & \textrm{if $n$ is even.}\end{array}\right.
\end{equation*}
We remark that the Lehmer sequence $(U_n)_{n \ge 1}$ satisfies the following recurrence relation over $R$:
$$
U_{n+4} = (\lambda^2 + \eta^2) U_{n+2} - (\lambda\eta)^2 U_n, \quad n = 1, 2, \ldots.
$$

\begin{theorem}  \label{thm:strong3}
The sequence $(U_n)_{n \ge 1}$ is a strong divisibility sequence.
\end{theorem}

\begin{theorem}  \label{thm:primitive3}
Assume that $R$ is the integral closure of a function field containing $\lambda$ and $\eta$, and let $d$ be the degree of $R$. 
Then in the sequence $(U_n)_{n \ge 1}$, each term $U_n$ has a primitive prime divisor whenever $n$ satisfies \eqref{eq:boundn}.
\end{theorem}  

\begin{remark}
If $R$ does not contain $\lambda$ and $\eta$, then 
 since $\lambda$ and $\eta$ satisfy the equation $x^4 - (\lambda^2 + \eta^2)x^2 + \lambda^2 \eta^2 = 0$ over $R$, 
we can get an extension of $R$ whose degree is at most $4d$ (noticing also $\lambda\eta \in R$) and which contains $\lambda$ and $\eta$, 
and so the result in Theorem~\ref{thm:primitive3} still holds if we replace $d$ by $4d$ in \eqref{eq:boundn}. 
\end{remark}

Finally, we explain the reason in the condition \eqref{eq:boundn} why we impose $p \nmid n$ when 
the characteristic $p >0$. We take the sequence $(U_n)_{n \ge 1}$ for example. 
First, for any even integer $m \ge 2$, we have
\begin{align*}
U_{mp} = \frac{\lambda^{mp} - \eta^{mp}}{\lambda^2 - \eta^2} & = \frac{(\lambda^{m} - \eta^{m})^p}{\lambda^2 - \eta^2} =
\Big(\frac{\lambda^{m} - \eta^{m}}{\lambda^2 - \eta^2}\Big)^p \cdot \frac{(\lambda^2 - \eta^2)^{p}}{\lambda^2 - \eta^2} \\
& = U_{m}^{p} \cdot  \frac{\lambda^{2p} - \eta^{2p}}{\lambda^2 - \eta^2} = U_{m}^{p} U_{2p}.
\end{align*}
Similarly, for any odd integer $m \ge 1$, we have
\begin{equation*}
U_{mp} = \left\{\begin{array}{ll}
 U_{m}^{p} U_{p} &\quad\text{if $p$ is odd},\\
 U_{m}^{p} &\quad\text{if $p = 2$}.
\end{array}\right.
\end{equation*}
Hence, for any integer $m \ge 2$, the term $U_{mp}$ does not have any primitive prime divisor in the sequence $(U_n)_{n \ge 1}$
when $p>0$. 

For the proofs, we first prove Theorems~\ref{thm:strong3} and \ref{thm:primitive3} in Section~\ref{sec:Lehmer},
and then prove Theorems~\ref{thm:strong2} and \ref{thm:primitive2} in Section~\ref{sec:Lucas},
and finally prove Theorems~\ref{thm:strong1} and \ref{thm:primitive1} in Section~\ref{sec:Zig}.
We collect some basic results in Section~\ref{sec:pre}.

\section{Preliminaries}
\label{sec:pre}

In this section, we collect some basic results for the proofs of the main results.
Recall that $R$ is a Dedekind domain.

We first recall some basic properties of the resultant of two homogeneous polynomials  in the following lemma; see \cite[Proposition 2.13]{Silverman}.

\begin{lemma}   \label{lem:Res}
Let
\begin{align*}
& A(x,y) = a_0 x^m + a_1 x^{m-1}y + \cdots + a_m y^m = a_0 \prod_{i=1}^{m} (x - \alpha_i y), \\
& B(x,y) = b_0 x^n + b_1 x^{n-1}y + \cdots + b_n y^n = b_0 \prod_{j=1}^{n} (x - \beta_j y)
\end{align*}
be two non-constant homogeneous polynomials in $x$ and $y$ of degrees $m$ and $n$ defined over a field.
Then, their resultant $\Res(A,B)$ equals to
$$
\Res(A,B) = a_0^n b_0^m \prod_{i=1}^{m} \prod_{j=1}^{n} (\alpha_i - \beta_j) \in \Z[a_0, \ldots, a_m, b_0, \ldots, b_n].
$$
Moreover, there exist four homogeneous polynomials $C_1, D_1, C_2, D_2$  in $x$ and $y$ and with coefficients in $\Z[a_0, \ldots, a_m, b_0, \ldots, b_n]$ such that
\begin{align*}
& C_1A + D_1 B = \Res(A,B)x^{m+n-1}, \\
& C_2A + D_2 B = \Res(A,B)y^{m+n-1}.
\end{align*}
\end{lemma}

For any integer $n \ge 1$, the $n$-th homogeneous cyclotomic polynomial is defined by
$$
\Phi_n (x,y) = \prod_{k=1, \, \gcd(k,n)=1}^{n} (x - \omega_n^k y) \in \Z[x,y],
$$
where $\omega_n$ is a primitive $n$-th root of unity in the complex numbers $\C$, and we also define the polynomial
$$
P_n(x,y) = \frac{x^n - y^n}{x-y} = \sum_{k=0}^{n-1} x^{n-1-k}y^k = \prod_{k=1}^{n-1} (x - \omega_n^k y)  \in \Z[x,y].
$$
Clearly,
$$
x^n - y^n = \prod_{j \mid n} \Phi_j (x,y), \qquad P_n (x,y) = \prod_{j \mid n, \, j \ge 2} \Phi_j (x,y).
$$
It is well-known that the degree of $\Phi_n(x,y)$ equals to $\varphi(n)$, where $\varphi$ is Euler's totient function.

The following lemma is \cite[Lemma 2.4]{FW} about the resultant of $P_m(x,y)$ and $P_n(x,y)$.

\begin{lemma}  \label{lem:Res2}
For any positive coprime  integers $m$ and $n$, we have $\Res(P_m, P_n) = \pm 1$.
\end{lemma}

The next lemma is a strong analogue of \cite[Lemma 2.3]{Sha}.
It is about coprime ideals (instead of coprime elements). 
Note that  coprime principal ideals imply coprime elements, but the converse is not true in general.

\begin{lemma}  \label{lem:coprime}
Let $a, b$ be two algebraic elements over $R$.
Assume that both $(a+b)^2$ and $ab$ are in $R$, and the two principal ideals $\langle (a+b)^2 \rangle$ and $\langle ab \rangle$ of $R$ are coprime.
Let $A(x,y), B(x,y) \in \Z[x,y]$ be non-constant homogeneous polynomials with resultant $\Res(A,B) = \pm 1$.
Assume that both $A(a,b)$ and $B(a,b)$ are in $R$.
Then, the principal ideals $\langle A(a,b) \rangle$ and $\langle B(a,b) \rangle$ of $R$ are coprime.
\end{lemma}

\begin{proof}
Let $m = \deg A$ and $n = \deg B$.
Note that both $(a+b)^2$ and $ab$ are in $R$.
So, by induction we obtain that for any  integer $k \ge 1$, $a^{2k} + b^{2k}$ is also in $R$.
By Lemma~\ref{lem:Res} and noticing $\Res(A,B) = \pm 1$, we have that there exist $u_1, w_1, u_2, w_2 \in \Z[a, b]$ such that
$$
u_1 A(a,b) + w_1B(a,b) = a^{2(m+n-1)} + b^{2(m+n-1)} \in R
$$
and
$$
u_2 A(a,b) + w_2 B(a,b) = a^{2(m+n-1)}b^{2(m+n-1)} \in R.
$$
Notice that $u_1, w_1, u_2, w_2$ might be not in $R$. 

By contradiction, suppose that the principal ideals $\langle A(a,b) \rangle$ and $\langle B(a,b) \rangle$ of $R$ are not coprime. Then, noticing also $R$ is a Dedekind domain, we know that there is a prime ideal, say $\q$, of $R$ such that both $A(a,b)$ and $B(a,b)$ are in $\q$. 

Now, let $M$ be the fractional field of $R$, let $M(a, b)$ be the field generated by $a$ and $b$ over $M$, 
and let $R^{\prime}$ be the integral closure of $R$ in $M(a, b)$. 
Since $a, b$ are two algebraic elements over $R$, we have that $M(a, b)$ is a finite extension of $M$. 
So, $R^{\prime}$ is also a Dedekind domain. 
By the Cohen–Seidenberg theorem (see \cite[Theorem 2]{CS}),  there is a prime ideal, say $\q^{\prime}$, of $R^{\prime}$ such that $\q^{\prime} \cap R = \q$. 
Clearly, both $A(a,b)$ and $B(a,b)$ are in $\q^{\prime}$. 

By assumption, we note that $a,b$ are integral over $R$ satisfying the equation $x^4 - (a^2 + b^2)x^2 + a^2b^2 = 0$. 
So, both $a$ and $b$ are in $R^{\prime}$. 
Then, the above $u_1, w_1, u_2, w_2$ are all in $R^{\prime}$. 
Hence, we get that $a^{2(m+n-1)} + b^{2(m+n-1)} \in \q^{\prime}$ and $a^{2(m+n-1)}b^{2(m+n-1)} \in \q^{\prime}$. 
Moreover, since both  $a^{2(m+n-1)} + b^{2(m+n-1)}$ and $a^{2(m+n-1)}b^{2(m+n-1)}$ (which equals to $(ab)^{2(m+n-1)}$) are in $R$, 
we know that they are both in $\q$ (because $\q^{\prime} \cap R =\q$). So, we have $ab \in \q$. 

Let $k = m+n-1$. We have shown that $\q \mid a^{2k} + b^{2k}$ and $\q \mid ab$ in $R$.
Note that
\begin{align*}
(a^2 + b^2)^k & = a^{2k} + b^{2k} + \sum_{i=1}^{k-1}\binom{k}{i}(a^2)^i (b^2)^{k-i}  \\
& =  a^{2k} + b^{2k} + a^2b^2 \sum_{i=1}^{k-1}\binom{k}{i}(a^2)^{i-1} (b^2)^{k-1-i},
\end{align*}
where $\sum_{i=1}^{k-1}\binom{k}{i}(a^2)^{i-1} (b^2)^{k-1-i}$ is in $R$
(because it is symmetric in $a^2$ and $b^2$, and both $a^2 + b^2$ and $a^2b^2$ are in $R$).
Hence, we obtain $\q \mid (a^2 + b^2)^{k}$, then $\q \mid a^2 + b^2$,  and thus $\q \mid (a + b)^2$ in $R$ (because $\q \mid ab$).
This leads to a contradiction with the assumption that the principal ideals $\langle (a+b)^2 \rangle$ and $\langle ab \rangle$ of $R$ are coprime.
Therefore,  the principal ideals $\langle A(a,b) \rangle$ and $\langle B(a,b) \rangle$ of $R$ are coprime.
\end{proof}

Using Lemma~\ref{lem:coprime}, we obtain the following five lemmas about coprime ideals in $R$ (instead of coprime elements in $R$), 
which are strong analogues of \cite[Lemmas 2.4--2.8]{Sha}. 
For the convenience of the readers, we also present their proofs.

\begin{lemma}  \label{lem:PmPn2}
Let $a, b$ be two algebraic elements over $R$.
Assume that both $a+b$ and $ab$ are in $R$, and the two principal ideals $\langle a+b \rangle$ and $\langle ab \rangle$ of $R$ are coprime.
Then, for any positive coprime integers $m, n$, the principal ideals $\langle P_m(a,b) \rangle$ and $\langle P_n(a,b) \rangle$ of $R$ are coprime.
\end{lemma}

\begin{proof}
Without loss of generality, we can assume $m \ge 2, n \ge 2$.
Clearly, both $P_m(a,b)$ and $P_n(a,b)$ are in $R$.
Because both $a+b$ and $ab$ are in $R$, and both $P_m(a,b)$ and $P_n(a,b)$ are symmetric in $a$ and $b$.

By assumption, the integers $m$ and $n$ are coprime, and the principal ideals $\langle (a+b)^2 \rangle$ and $\langle ab \rangle$ of $R$ are coprime.
Then, the desired result follows directly from Lemmas~\ref{lem:Res2} and \ref{lem:coprime}.
\end{proof}

\begin{lemma}  \label{lem:PmPn-odd}
Let $a, b$ be defined as in Lemma~\ref{lem:coprime}.
Let $m,n$ be two positive coprime  integers such that both $m$ and $n$ are odd.
Then, the principal ideals $\langle P_m(a,b) \rangle$ and $\langle P_n(a,b) \rangle$ of $R$ are coprime.
\end{lemma}

\begin{proof}
Without loss of generality, we can assume $m \ge 3, n \ge 3$.
By assumption, both $a^2 + b^2$ and $ab$ are in $R$.
Moreover, we have that for any integer $k \ge 1$, $a^{2k} + b^{2k} \in R$.

Note that the polynomial $P_m(x,y)$ is homogeneous of even degree $m-1$ (since $m$ is odd) and symmetric in $x$ and $y$.
So, if $x^iy^j$ is a term in $P_m(x,y)$, then $x^jy^i$ is also a term in $P_m(x,y)$, and then assuming $i \le j$, we have
$$
a^i b^j + a^j b^i = (ab)^i (a^{j-i} + b^{j-i}) \in R,
$$
where we use the fact that $j - i$ is even (because $i + j = m-1$ is even).
Hence, we have that $P_m(a,b)$ is in $R$.
Similarly, $P_n(a,b)$ is also in $R$, because $n$ is also odd.

Now, the desired result follows directly from Lemmas~\ref{lem:Res2} and \ref{lem:coprime}.
\end{proof}

\begin{lemma}  \label{lem:PmPn-mix}
Let $a, b$ be defined as in Lemma~\ref{lem:coprime}.
Let $m,n$ be two positive coprime  integers such that  $m$ is odd and $n$ is even.
Then, the principal ideals $\langle P_m(a,b) \rangle$ and $\langle P_n(a,b)/(a + b) \rangle$ of $R$ are coprime.
\end{lemma}

\begin{proof}
Without loss of generality, we can assume $m \ge 3, n \ge 2$.
Since $m$ is odd, as in the proof of Lemma~\ref{lem:PmPn-odd} $P_m(a,b)$ is in $R$.
For any odd integer $k \ge 1$, note that
$$
\frac{a^k + b^k}{a+b} = a^{k-1} - a^{k-2}b + \cdots - ab^{k-2} + b^{k-1}
$$
is homogeneous of even degree $k-1$ and is symmetric in $a$ and $b$, so it is in $R$.
Hence, for even $n$, since
\begin{align*}
\frac{P_n(a,b)}{a + b} = \frac{a^{n-1} + b^{n-1}}{a+b}  + ab \cdot \frac{a^{n-3} + b^{n-3}}{a+b} + \cdots
 + a^{\frac{n-2}{2}}b^{\frac{n-2}{2}} \cdot \frac{a + b}{a+b},
\end{align*}
we have that $P_n(a,b)/(a + b)$ is in $R$.

Denote $T_n(x,y) = P_n(x,y)/(x+y)$, which can be viewed as a polynomial over $\Z$.
It has been shown in the proof of \cite[Lemma 2.6]{Sha} that the resultant of $P_m(x,y)$ and $T_n(x,y)$ is equal to $\pm 1$.

Hence, using Lemma~\ref{lem:coprime} we obtain that the principal ideals $\langle P_m(a,b) \rangle$ and $\langle P_n(a,b)/(a + b) \rangle$ of $R$ are coprime.
\end{proof}

\begin{lemma}  \label{lem:Pmn}
Let $a, b$ be defined as in Lemma~\ref{lem:coprime}.
Let $m, n$ be two positive integers such that both $m$ and $n$ are odd.
Then, the principal ideals $\langle P_m(a^n,b^n) \rangle$ and $\langle (a^n + b^n)/(a + b) \rangle$ of $R$ are coprime.
\end{lemma}

\begin{proof}
Without loss of generality, we can assume $m \ge 3, n \ge 3$.
As before, since $m$ and $n$ are odd, $P_m(a^n,b^n)$ and $(a^n + b^n)/(a + b)$ are indeed in $R$.

Define
$$
V_m(x,y) =P_m(x^n, y^n), \qquad  W_n(x,y) = \frac{x^n + y^n}{x+y}.
$$
It has been shown in the proof of \cite[Lemma 2.7]{Sha} that the resultant  $\Res(V_m, W_n) = \pm 1$.
So, the desired result follows directly from Lemma~\ref{lem:coprime}.
\end{proof}

\begin{lemma}  \label{lem:abn}
Let $a, b$ be defined as in Lemma~\ref{lem:coprime}.
Then, for any odd integer $n \ge 1$, the principal ideals  $\langle P_n(a,b) \rangle$ and $\langle (a+b)^2 \rangle$ of $R$  are coprime.
\end{lemma}

\begin{proof}
Without loss of generality, we fix an odd integer $n \ge 3$.
Since $n$ is odd, by Lemma~\ref{lem:PmPn-odd} we have that $P_n(a,b)$ is  in $R$.
Moreover, it has been shown in the proof of \cite[Lemma 2.8]{Sha} that the resultant of the polynomials $P_n(x,y)$ and $(x+y)^2$ is equal to $\pm 1$.
Hence, the desired result follows from Lemma~\ref{lem:coprime}.
\end{proof}

\section{Proofs of Theorems~\ref{thm:strong3} and \ref{thm:primitive3}}
\label{sec:Lehmer}

We first prove Theorem~\ref{thm:strong3}. 

\begin{proof}[Proof of Theorem~\ref{thm:strong3}]
Let $j = \gcd(m,n)$. Then, $m/j$ and $n/j$ are coprime. 

First, assume that both $m$ and $n$ are even. Then, $j$ is also even.
Clearly,
$$
U_m = U_j P_{m/j}(\lambda^j, \eta^j), \quad U_n = U_j P_{n/j}(\lambda^j, \eta^j).
$$
By assumption, both  $\lambda^j + \eta^j$ and $\lambda^j \eta^j$ are in $R$, and moreover the principal ideals $\langle \lambda^j + \eta^j \rangle$ and $\langle \lambda^j \eta^j \rangle$ of $R$ are coprime
(as in the last paragraph of the proof of Lemma~\ref{lem:coprime}).
Hence, it follows from Lemma~\ref{lem:PmPn2} that $P_{m/j}(\lambda^j, \eta^j)$ and $P_{n/j}(\lambda^j, \eta^j)$ are coprime in $R$,
and thus, $\gcd(U_m, U_n) = U_j$.

Now, assume that both $m$ and $n$ are odd. Then, $j$ is also odd.
Clearly,
$$
U_m = U_j P_{m/j}(\lambda^j, \eta^j), \quad U_n = U_j P_{n/j}(\lambda^j, \eta^j).
$$
Note that both $(\lambda^j + \eta^j)^2$ and $\lambda^j \eta^j$ are in $R$, and moreover the principal ideals $\langle (\lambda^j + \eta^j)^2 \rangle$ and $\langle \lambda^j \eta^j \rangle$ of $R$ are coprime.
So, it follows from Lemma~\ref{lem:PmPn-odd} that $P_{m/j}(\lambda^j, \eta^j)$ and $P_{n/j}(\lambda^j, \eta^j)$ are coprime in $R$,
and so, $\gcd(U_m, U_n) = U_j$.

Finally, without loss of generality, we assume that  $m$ is odd and $n$ is even.
Then, $j$ is odd. Clearly,
\begin{align*}
& U_m = U_j P_{m/j}(\lambda^j, \eta^j), \\
& U_n = U_j \cdot \frac{P_{n/j}(\lambda^j, \eta^j)}{\lambda^j + \eta^j} \cdot \frac{\lambda^j + \eta^j}{\lambda + \eta}.
\end{align*}
Using Lemma~\ref{lem:PmPn-mix} we have that $P_{m/j}(\lambda^j, \eta^j)$ and $P_{n/j}(\lambda^j, \eta^j)/(\lambda^j + \eta^j)$ are coprime in $R$.
In addition, by Lemma~\ref{lem:Pmn} we obtain that $P_{m/j}(\lambda^j, \eta^j)$ and $(\lambda^j + \eta^j)/(\lambda + \eta)$ are coprime in $R$.
Hence, $\gcd(U_m, U_n) = U_j$.
This completes the proof.
\end{proof}

To prove Theorem~\ref{thm:primitive3}, we need to make some more preparations.

Recall that $p$ is the characteristic of $R$.
As usual, denote by $v_\q(h)$ the maximal power to which a prime ideal $\q$ (of $R$) divides $h \in R$.

In this section, let $M$ be the fraction field of $R$.
The elements $\lambda$ and $\eta$ have been defined at the beginning of Section~\ref{sec:main}.
By assumption, $M(\lambda)$ is an algebraic extension generated by $\lambda$ over $M$ and also containing $\eta$.
For any prime ideal $\q$ of $R$, as usual $v_\q$ induces a non-Archimedean valuation of $M$. 
In fact, $v_\q(\alpha) \ge 0$ for any $\alpha \in R$, and $\q = \{\alpha \in R: \, v_\q(\alpha) >0\}$. 
This valuation can be extended to the field $M(\lambda)$ (still denoted by $v_\q$); see, for instance, \cite[Theorem 3.1.2]{EP}.

\begin{lemma}  \label{lem:vU}
Let $\q$ be a prime ideal of $R$ dividing $U_n$ for some $n \ge 3$.
Then, for any integer $m \ge 1$, if either $p=0$ or $p \nmid m$, we have
$v_\q (U_{mn}) = v_\q (U_n)$.
\end{lemma}

\begin{proof}
First, note that $\lambda, \eta$ are both integral over the ring $R$
(because they satisfy the equation $x^4 - (\lambda^2 + \eta^2)x^2 + \lambda^2 \eta^2 = 0$ over $R$).
So, we have $v_\q(\lambda) \ge 0$ and  $v_\q(\eta) \ge 0$.

Assume that $v_\q(\eta) > 0$.
Note that by definition, either $\lambda^n = \eta^n + (\lambda - \eta) U_n$, or $\lambda^n = \eta^n + (\lambda^2 - \eta^2) U_n$.
We also note that $v_\q (U_n) > 0$ by assumption.
So, we obtain $v_\q(\lambda^n) > 0$,
which gives $v_\q (\lambda) > 0$.
Thus,
$$
v_\q(\lambda + \eta) >0, \qquad v_\q(\lambda\eta) > 0.
$$
This contradicts the assumption that
the principal ideals $\langle (\lambda + \eta)^2 \rangle$ and $\langle \lambda \eta \rangle$ of $R$ are coprime.
So, we must have  $v_\q(\eta) = 0$.

Now, suppose that $n$ is odd.
Then, $U_n = (\lambda^n - \eta^n)/(\lambda - \eta)$, and so
$$
\lambda^n = \eta^n + (\lambda - \eta)U_n.
$$
Then,  we have
$$
\lambda^{mn} = \big( \eta^n +  (\lambda - \eta)U_n \big)^m
= \eta^{mn} + \sum_{i=1}^{m} \binom{m}{i} (\lambda - \eta)^i U_n^{i} \eta^{n(m-i)},
$$
which implies
$$
\frac{\lambda^{mn} - \eta^{mn}}{\lambda - \eta}
= m\eta^{n(m-1)}U_n + \sum_{i=2}^{m}\binom{m}{i} (\lambda - \eta)^{i-1}\eta^{n(m-i)} U_n^{i}.
$$
Hence, we have that for odd $m$
$$
U_{mn} = m\eta^{n(m-1)}U_n + \sum_{i=2}^{m}\binom{m}{i} (\lambda - \eta)^{i-1}\eta^{n(m-i)} U_n^{i},
$$
and for even $m$
$$
(\lambda + \eta)U_{mn} = m\eta^{n(m-1)}U_n + \sum_{i=2}^{m}\binom{m}{i} (\lambda - \eta)^{i-1}\eta^{n(m-i)} U_n^{i}.
$$
Besides, since $n$ is odd and $v_\q (U_n) > 0$, using Lemma~\ref{lem:abn} we obtain $v_\q (\lambda + \eta) = 0$.
Hence, the desired result follows from the ultrametric inequality for non-Archimedean valuations.
Here we need to use $v_\q(\eta) = 0$ and also $v_\q(m) = 0$ (since we are in the function field setting and $p \nmid m$ if $p>0$).

Finally, assume that $n$ is even.
Then, as the above, for any integer $m \ge 1$ we have
$$
U_{mn} = m\eta^{n(m-1)}U_n + \sum_{i=2}^{m}\binom{m}{i} (\lambda^2 - \eta^2)^{i-1}\eta^{n(m-i)} U_n^{i}.
$$
The desired result now follows similarly.
\end{proof}

Now, we need to investigate some properties of the elements  $\Phi_n(\lambda,\eta)$ with $n \ge 3$.

\begin{lemma}  \label{lem:Phin}
For any integer $n \ge 3$, $\Phi_n(\lambda, \eta)$ is in $R$. 
\end{lemma}
\begin{proof}
Note that for any integer $n \ge 3$, the polynomial $\Phi_n(x,y)$ is homogeneous of even degree $\varphi(n)$ and is symmetric in $x$ and $y$.
Then, as in the proof of Lemma~\ref{lem:PmPn-odd}, we have that $\Phi_n(\lambda,\eta)$ is in $R$.
\end{proof}

The following lemma asserts that when the characteristic $p=0$ or $p \nmid n$,
any non-primitive prime divisor of the term $U_n$ does not divide $\Phi_n(\lambda,\eta)$. 

\begin{lemma}  \label{lem:nonprimitive}
Let $n$ be an integer such that $n \ge 3$ and $p \nmid n$ if $p > 0$.
Then, any non-primitive prime divisor of the term $U_n$ does not divide $\Phi_n(\lambda,\eta)$.
\end{lemma}

\begin{proof}
As in \cite{Ward}, we define the sequence $(Q_k)_{k \ge 1}$ of polynomials by $Q_1 = 1, Q_2 = 1$, and
$$
Q_k(x,y) = \Phi_k(x,y), \quad k = 3, 4, \ldots.
$$
Clearly, for integer $n \ge 3$, we have
$$
U_n = \prod_{j \mid n} Q_j(\lambda, \eta).
$$
From Lemma~\ref{lem:Phin}, the elements $Q_j(\lambda, \eta),  j\ge 1$ are all in $R$. 
By the M{\"o}bius inversion we obtain
\begin{equation}  \label{eq:Qn}
Q_n(\lambda, \eta) = \prod_{j\mid n} U_j^{\mu(n/j)}, 
\end{equation}
where $\mu$ is the M{\"o}bius function. 

Now, suppose that the characteristic $p > 0$.
Let $\q$ be a prime divisor of $U_n$ which is not primitive, where $p \nmid n$ by assumption.
Let $m$ be the smallest positive integer such that $\q \mid U_m$.
Noticing the choice of $\q$, we have $m < n$.
Moreover, using Theorem~\ref{thm:strong3} we obtain that $\gcd(U_m, U_n) = U_{\gcd(m,n)}$, and so $m \mid n$ (noticing the choice of $m$).
Similarly, for any positive integer $j$, we obtain $v_\q(U_j) = 0$ if $m \nmid j$.
Besides, by Lemma~\ref{lem:vU}, for any integer $j$ with $p \nmid j$ we have
$$
v_\q(U_{jm}) = v_\q(U_m).
$$
Hence, using \eqref{eq:Qn} and noticing $p \nmid n$, we obtain
\begin{align*}
v_\q(Q_n(\lambda,\eta)) & = \sum_{j\mid n} \mu(n/j) v_\q(U_j) \\
& = \sum_{jm \mid n} \mu(n/(jm)) v_\q(U_{jm}) \\
& = \sum_{j \mid n/m} \mu(n/(jm)) v_\q(U_m) \\
& = v_\q(U_m) \sum_{j \mid n/m} \mu(n/(jm)) = 0, 
\end{align*}
where the last identity follows from the well-known fact $\sum_{j \mid k} \mu(j) = 0$ for any integer $k \ge 2$.
So, any non-primitive prime divisor of $U_n$ does not divide $Q_n(\lambda,\eta)$
(that is, $\Phi_n(\lambda,\eta)$, since $n \ge 3$).

The proof for the case $p = 0$ follows exactly the same way.
\end{proof}

The following lemma is a further property about the elements  $\Phi_n(\lambda, \eta)$ 
when $R$ is the integral closure of a function field. 
Recall that $p$ is the characteristic of $R$, and $\varphi$ is Euler's totient function.

\begin{lemma}  \label{lem:nonunit}
Let $R$ be the integral closure of a function field, and let $d$ be the degree of $R$. 
Assume that both $\lambda$ and $\eta$ are in $R$. 
Then, $\Phi_n(\lambda, \eta)$ is not a unit in $R$ whenever $n$ satisfies \eqref{eq:boundn}.
\end{lemma}
\begin{proof}
We assume that $R$ is the integral closure of a function field $E$. 
So, $E$ is a finite extension of $K(x)$. 
Let $d$ be the degree of $R$ (that is, $d = [E:K(x)]$).  
Let $\overline{K}$ be the algebraic closure of $K$ contained in the fixed algebraic closure of $E$. 
As usual, let $\overline{K}E$ be the composition of the fields $\overline{K}$ and $E$. 
Put $d^{\prime} = [\overline{K}E : \overline{K}(x)]$.
Then, we have 
\begin{equation}  \label{eq:degree}
d^{\prime} \le d.
\end{equation}
Note that $R$ is the integral closure of $K[x]$ in $E$. 
Let $R^{\prime}$ be the integral closure of $\overline{K}[x]$ in $\overline{K}E$. 
Then, $R \subseteq R^{\prime}$. 

Let $\mathbb{U}^{\prime}$ be the group of units in $R^{\prime}$. 
Let $r^{\prime}$ be the rank of  $\mathbb{U}^{\prime}$ modulo $\overline{K}^*$. 
Note that the function field $\overline{K}E$ (over $\overline{K}(x)$) has at most $d^{\prime}$ valuations at infinity. Combining this (see \cite[Corollary 1 in page 243]{Rosen}) with \eqref{eq:degree}, we have 
\begin{equation}  \label{eq:rank}
r^{\prime} \le d^{\prime} - 1 \le d-1. 
\end{equation}

By the assumption on $\lambda$ and $\eta$ at the beginning of Section~\ref{sec:main}, 
we know that the principal ideals $\langle (\lambda + \eta)^2 \rangle$ and $\langle \lambda\eta \rangle$ of $R$ are  coprime ideals which are not both equal to $R$.
Then, $(\lambda + \eta)^2$ and $\lambda\eta$ can not be both in $K$. 
Moreover, at least one of $\lambda$ and $\eta$ is transcendental over $K$ (that is, not in $\overline{K}$), 
and 
\begin{equation}  \label{eq:notin}
\lambda/ \eta \not\in \overline{K}.
\end{equation}
Indeed, if $\lambda/ \eta \in \overline{K}$, then $\lambda = a\eta$ for some constant $a \in \overline{K}$, 
and then the principal ideals $\langle (\lambda + \eta)^2 \rangle$ and $\langle \lambda\eta \rangle$ of $R$ are not coprime, which contradicts with our assumption. 

For the integer $n$ satisfying \eqref{eq:boundn}, we have either $p=0$ or $p \nmid n$, and so there exist primitive $n$-th roots of unity in $\overline{K}$. 
For simplicity, we denote by $\zeta$ a primitive $n$-th root of unity in $\overline{K}$. 
We also define the subset $T$ of integers as follows:
\begin{equation*}
T = \{k \in \Z:\, 1\le k \le n, \gcd(k,n)=1\}. 
\end{equation*}
Clearly,  $|T| = \varphi(n)$.  
Then, by definition, we have
$$
\Phi_n(\lambda, \eta) = \prod_{k \in T} (\lambda - \zeta^k \eta).
$$
Notice that we have assumed either $p=0$, or $p > 0$ and $p \nmid n$, and so we have 
\begin{equation} \label{eq:zetaij}
\zeta^i \ne \zeta^j  \quad \textrm{for any distinct $i,j \in T$}.
\end{equation}

Now, we consider the unit equation 
\begin{equation} \label{eq:uniteq}
u_1 + u_2 =1 \quad \textrm{ with $(u_1, u_2) \in (\mathbb{U}^{\prime})^2$}.
\end{equation}
A solution $(u_1, u_2)$ of \eqref{eq:uniteq} is said to be \textit{trivial} if $u_1/u_2 \in \overline{K}^*$. 
Clearly, a solution $(u_1, u_2)$ is trivial if and only if $u_1, u_2 \in \overline{K}^*$.
When the characteristic $p > 0$, two distinct solutions $(u_1, u_2), (u_1^{\prime}, u_2^{\prime})$  of \eqref{eq:uniteq} are said to be \textit{equivalent}  
if there exists a positive integer $m$ such that
$$
\textrm{either } u_1 = (u_1^{\prime})^{p^m}, u_2 = (u_2^{\prime})^{p^m}, 
\qquad \textrm{or } u_1^{\prime} = u_1^{p^m}, u_2^{\prime} = u_2^{p^m}.
$$

In this paragraph, we assume that $\Phi_n(\lambda, \eta)$ is a unit in $R$ and $\varphi(n) \ge 3$. 
Then, we try to construct some non-trivial and also  inequivalent (if $p>0$) solutions to \eqref{eq:uniteq}. 
By definition, $\Phi_n(\lambda, \eta)$ is also a unit in $R^{\prime}$. 
Since $\lambda, \eta$ and $\zeta$ are all in $R^{\prime}$, we know that $\lambda - \zeta^k \eta$ is also a unit  in $R^{\prime}$ for each $k \in T$. 
Then, we choose three distinct integers $i, j, k \in T$ and construct the following solution to $\eqref{eq:uniteq}$:
\begin{equation} \label{eq:ijk}
\left( \frac{\zeta^k - \zeta^j}{\zeta^i - \zeta^j} \cdot \frac{\lambda - \zeta^i \eta}{\lambda - \zeta^k \eta}, \quad
\frac{\zeta^i - \zeta^k}{\zeta^i - \zeta^j} \cdot \frac{\lambda - \zeta^j \eta}{\lambda - \zeta^k \eta }\right),
\end{equation}
where the solution is well-defined due to \eqref{eq:notin} and \eqref{eq:zetaij}. 
In view of \eqref{eq:notin} we know that the solution \eqref{eq:ijk} is non-trivial. 
In addition, we choose another three distinct integers $i^{\prime}, j^{\prime}, k^{\prime} \in T$ and construct another non-trivial solution to \eqref{eq:uniteq}:
\begin{equation}  \label{eq:ijk'}
\left( \frac{\zeta^{k^{\prime}} - \zeta^{j^{\prime}}}{\zeta^{i^{\prime}} - \zeta^{j^{\prime}}} \cdot \frac{\lambda - \zeta^{i^{\prime}} \eta}{\lambda - \zeta^{k^{\prime}} \eta}, \quad
\frac{\zeta^{i^{\prime}} - \zeta^{k^{\prime}}}{\zeta^{i^{\prime}} - \zeta^{j^{\prime}}} \cdot \frac{\lambda - \zeta^{j^{\prime}} \eta}{\lambda - \zeta^{k^{\prime}} \eta} \right). 
\end{equation}
Now, assume that the solutions \eqref{eq:ijk} and \eqref{eq:ijk'} are the same. Then, we have 
\begin{equation}  \label{eq:solu1}
\frac{\zeta^k - \zeta^j}{\zeta^i - \zeta^j} \cdot \frac{\lambda - \zeta^i \eta}{\lambda - \zeta^k \eta}
=
\frac{\zeta^{k^{\prime}} - \zeta^{j^{\prime}}}{\zeta^{i^{\prime}} - \zeta^{j^{\prime}}} \cdot \frac{\lambda - \zeta^{i^{\prime}} \eta}{\lambda - \zeta^{k^{\prime}} \eta}
\end{equation}
and 
\begin{equation}  \label{eq:solu2}
\frac{\zeta^i - \zeta^k}{\zeta^i - \zeta^j} \cdot \frac{\lambda - \zeta^j \eta}{\lambda - \zeta^k \eta}
=
\frac{\zeta^{i^{\prime}} - \zeta^{k^{\prime}}}{\zeta^{i^{\prime}} - \zeta^{j^{\prime}}} \cdot \frac{\lambda - \zeta^{j^{\prime}} \eta}{\lambda - \zeta^{k^{\prime}} \eta}. 
\end{equation} 
In \eqref{eq:solu1}, if 
$$
\frac{\zeta^k - \zeta^j}{\zeta^i - \zeta^j} \ne  \frac{\zeta^{k^{\prime}} - \zeta^{j^{\prime}}}{\zeta^{i^{\prime}} - \zeta^{j^{\prime}}}, 
$$
then from \eqref{eq:solu1} we can get that $\lambda/\eta$ satisfies a quadratic equation with coefficients in $\overline{K}$, and so $\lambda/\eta \in \overline{K}$, which contradicts with  \eqref{eq:notin}. 
Hence, we must have 
$$
\frac{\zeta^k - \zeta^j}{\zeta^i - \zeta^j} = \frac{\zeta^{k^{\prime}} - \zeta^{j^{\prime}}}{\zeta^{i^{\prime}} - \zeta^{j^{\prime}}}. 
$$
Similarly, from \eqref{eq:solu2} we must have 
$$
\frac{\zeta^i - \zeta^k}{\zeta^i - \zeta^j}  = \frac{\zeta^{i^{\prime}} - \zeta^{k^{\prime}}}{\zeta^{i^{\prime}} - \zeta^{j^{\prime}}}. 
$$
So, we have 
$$
\frac{\lambda - \zeta^i \eta}{\lambda - \zeta^k \eta} = \frac{\lambda - \zeta^{i^{\prime}} \eta}{\lambda - \zeta^{k^{\prime}} \eta}, 
\qquad 
\frac{\lambda - \zeta^j \eta}{\lambda - \zeta^k \eta} = \frac{\lambda - \zeta^{j^{\prime}} \eta}{\lambda - \zeta^{k^{\prime}} \eta}.  
$$
We rewrite the above as 
$$
\frac{\lambda/\eta - \zeta^i }{\lambda/\eta - \zeta^k} = \frac{\lambda/\eta - \zeta^{i^{\prime}}}{\lambda/\eta - \zeta^{k^{\prime}}}, 
\qquad 
\frac{\lambda/\eta - \zeta^j}{\lambda/\eta - \zeta^k} = \frac{\lambda/\eta - \zeta^{j^{\prime}}}{\lambda/\eta - \zeta^{k^{\prime}}};  
$$
and so we have 
$$
(\zeta^i + \zeta^{k^{\prime}}) \frac{\lambda}{\eta} - \zeta^{i + k^\prime} = (\zeta^{i^\prime} + \zeta^k) \frac{\lambda}{\eta} - \zeta^{i^\prime + k}
$$
and 
$$
(\zeta^j + \zeta^{k^{\prime}}) \frac{\lambda}{\eta} - \zeta^{j + k^\prime} = (\zeta^{j^\prime} + \zeta^k) \frac{\lambda}{\eta} - \zeta^{j^\prime + k}.
$$
Then, combining this with \eqref{eq:notin}, we obtain 
$$
\zeta^{i+ k^{\prime}} = \zeta^{i^{\prime}+k},
\qquad   
\zeta^i + \zeta^{k^{\prime}} = \zeta^{i^{\prime}} + \zeta^k, 
$$
and
$$
\zeta^{j+ k^{\prime}} = \zeta^{j^{\prime}+k},
\qquad   
\zeta^j + \zeta^{k^{\prime}} = \zeta^{j^{\prime}} + \zeta^k. 
$$
Thus, we get $\zeta^{i- j} = \zeta^{i^{\prime}- j^{\prime}}$ (which is not equal to 1 due to \eqref{eq:zetaij}) and 
$$
\zeta^i - \zeta^j = \zeta^{i^{\prime}} - \zeta^{j^{\prime}}, 
\textrm{ that is }
\zeta^j(\zeta^{i-j} -1)  = \zeta^{j^{\prime}}(\zeta^{i^{\prime} -j^{\prime}} - 1), 
$$
which implies $\zeta^j = \zeta^{j^{\prime}}$, and so (noticing $j, j^{\prime} \in T$ and using \eqref{eq:zetaij})
$$
j = j^{\prime}. 
$$
Then, we further obtain 
$$
k = k^{\prime}, \qquad i = i^{\prime}. 
$$
So, we have 
$$
(i,j,k) = (i^{\prime},j^{\prime},k^{\prime}).
$$
Hence, the solutions \eqref{eq:ijk} and \eqref{eq:ijk'} are the same if and only if 
\begin{equation}  \label{eq:ijk=}
(i,j,k) = (i^{\prime},j^{\prime},k^{\prime}).
\end{equation}
Therefore, if $\Phi_n(\lambda, \eta)$ is a unit in $R$, we can construct at least 
\begin{equation}  \label{eq:solution}
\varphi(n)(\varphi(n)-1)(\varphi(n)-2)
\end{equation}
non-trivial solutions of the form \eqref{eq:ijk} to the unit equation \eqref{eq:uniteq}. 

Furthermore, when $p>0$, the solutions \eqref{eq:ijk} and \eqref{eq:ijk'} are inequivalent 
(by definition, applying similar arguments as the above and using \eqref{eq:notin} and \eqref{eq:zetaij}). 
Therefore, if $\Phi_n(\lambda, \eta)$ is a unit in $R$ and $p>0$, we can construct at least 
\begin{equation}  \label{eq:solution'}
\varphi(n)(\varphi(n)-1)(\varphi(n)-2)
\end{equation}
inequivalent non-trivial solutions of the form \eqref{eq:ijk} for the unit equation \eqref{eq:uniteq}. 

However, when the characteristic $p=0$, by \cite[the discussion below (1.5)]{EZ} (see also \cite[Theorem 7.4.1]{EG} or \cite[Corollary 4]{Zannier}), 
we deduce that the unit equation \eqref{eq:uniteq} has at most 
\begin{equation} \label{eq:solution1}
3^{2r^{\prime}} \textrm{ non-trivial solutions $(u_1, u_2) \in (\mathbb{U}^{\prime})^2$.}
\end{equation}
So, combining \eqref{eq:solution1}, \eqref{eq:solution} with \eqref{eq:rank} and in view of \eqref{eq:boundn}, we get that when $p=0$ and if 
\begin{equation*}
\varphi(n)(\varphi(n)-1)(\varphi(n)-2) > 3^{2d-2}, 
\end{equation*}
then $\Phi_n(\lambda, \eta)$ is not a unit. This completes the proof for the case when $p=0$.

In addition, when the characteristic $p > 0$, 
by Theorem 1.2 in \cite{KP} (see also \cite[Theorem 2.5.2]{Koy} for more details), we know that the unit equation \eqref{eq:uniteq} has at most 
\begin{equation} \label{eq:solution2}
31 \cdot 19^{2r^{\prime}} \textrm{ inequivalent non-trivial solutions $(u_1, u_2) \in (\mathbb{U}^{\prime})^2$.}
\end{equation}
So, combining \eqref{eq:solution2}, \eqref{eq:solution'} with \eqref{eq:rank} and in view of \eqref{eq:boundn}, we get the desired result for the case when $p > 0$. 
\end{proof}

\begin{remark}
We don't know whether it can hold that for some $\lambda$ and $\eta$,  $\Phi_n(\lambda, \eta)$ is not a unit for any $n \ge 3$. 
However, for any $n \ge 3$, we first fix a unit (say $u$) and such an element $\eta$ in $R$, and then find such an element $\lambda$ by solving the equation $\Phi_n(\lambda, \eta) = u$ (assuming $R$ is large enough to contain $\lambda$), and so in this case  $\Phi_n(\lambda, \eta)$ is a unit.
\end{remark}

Now, we are ready to prove Theorem~\ref{thm:primitive3}.

\begin{proof}[Proof of Theorem~\ref{thm:primitive3}]
Recall that $Q_1 = 1, Q_2 = 1$, and $Q_k(x,y) = \Phi_k(x,y), \quad k = 3, 4, \ldots$.
For the integer $n$ satisfying \eqref{eq:boundn}, we have $n \ge 3$, and so we have
$$
U_n = \prod_{j \mid n} Q_j(\lambda, \eta).
$$
Thus, any primitive prime divisor of the term $U_n$  must divide $Q_n(\lambda, \eta)$ (that is, $\Phi_n(\lambda, \eta)$).

In addition, by Lemma~\ref{lem:nonprimitive} we know that
 any non-primitive prime divisor of $U_n$ does not divide $\Phi_n(\lambda,\eta)$.

Hence, the primitive prime divisors of $U_n$ are exactly the prime divisors of
 $\Phi_n(\lambda,\eta)$ (if they exist).

Note that here $R$ is the integral closure of a function field.
So, when the integer $n$ satisfies \eqref{eq:boundn}, by Lemmas~\ref{lem:Phin} and \ref{lem:nonunit}  we know that $\Phi_n(\lambda, \eta)$ is in $R$ and it is not a unit in $R$. So, $\Phi_n(\lambda,\eta)$ has a prime divisor in $R$.

Therefore, collecting the above deductions, we conclude that
the term  $U_n$ has a primitive prime divisor when $n$ satisfies \eqref{eq:boundn}.
\end{proof}

\begin{remark}   \label{rem:Un}
In the proof of Theorem~\ref{thm:primitive3}, we obtain more:
when $R$ is the integral closure of a function field, the \textit{primitive part}
(that is, the product of all the primitive prime divisors to their respective powers)
of $U_n$ is exactly the principal ideal $\langle \Phi_n(\lambda, \eta) \rangle$, where $n$ satisfies \eqref{eq:boundn}.
\end{remark}

\begin{remark}
In Theorem~\ref{thm:primitive3} (and also in Theorems~\ref{thm:primitive1} and \ref{thm:primitive2}), 
we only consider the case of function fields of transcendental degree one over $K$. 
In fact, in general a function field is defined as a finite extension of $K(x_1, \ldots, x_h)$ (not necessarily for $h=1$), where $x_1, \ldots, x_h$ are independent variables over $K$. 
Indeed, both the results in \cite{EZ} and \cite{KP} still hold in this general case.  
But, in the above proof of Theorem~\ref{thm:primitive3}, we need Lemma~\ref{lem:nonprimitive} whose proof relies on the condition that $R$ is a Dedekind domain (because Theorem~\ref{thm:strong3} is used there). 
Moreover, the proof of Theorem~\ref{thm:strong3} requires  Lemma~\ref{lem:coprime}, which indeed needs $R$ ot be a Dedekind domain. However, when $R$ is the integral closure of a general function field, $R$ might be not a Dedekind domain; for instance $R=K[x_1, \ldots, x_h], h \ge 2$ (this case in fact was studied in \cite{Sha}). 
\end{remark}

\section{Proofs of Theorems~\ref{thm:strong2} and \ref{thm:primitive2}}
\label{sec:Lucas}

Both proofs follow from the same arguments as in \cite[Section 4]{Sha}
with only minor changes and from Theorems~\ref{thm:strong3} and \ref{thm:primitive3}.

\begin{proof}[Proof of Theorem~\ref{thm:strong2}]
Fix positive integers $m, n$ with $j = \gcd(m,n)$.
If either both $m,n$ are even, or both $m,n$ are odd,
it follows directly from Theorem~\ref{thm:strong3} (choosing $\lambda = \alpha, \eta = \beta$ there) that $\gcd(L_m, L_n) = L_j$. Indeed, when both $m$ and $n$ are even, $j$ is also even, and then by choosing $\lambda=\alpha$ and $\eta=\beta$ we have that $L_m=(\alpha+\beta)U_m$, $L_n=(\alpha+\beta)U_n$ and $L_j=(\alpha+\beta)U_j$; and moreover by Theorem~\ref{thm:strong3} we have $\gcd(U_m, U_n)=U_j$, and thus we have $(\alpha+\beta)\gcd(U_m, U_n)=(\alpha+\beta)U_j$, that is, $\gcd(L_m, L_n)=L_j$. 

Now, without loss of generality, assume that $m$ is even and $n$ is odd.
Then, $j$ is odd. By Theorem~\ref{thm:strong3} we obtain
$$
\gcd\Big(\frac{\alpha^m - \beta^m}{\alpha^2 - \beta^2}, \frac{\alpha^n - \beta^n}{\alpha - \beta} \Big)
= \frac{\alpha^j - \beta^j}{\alpha - \beta}.
$$
Using Lemma~\ref{lem:abn} we know that $(\alpha^n - \beta^n)/(\alpha - \beta)$ and $\alpha + \beta$ are coprime in $R$.
Hence, we obtain $\gcd(L_m, L_n) = L_j$.
This completes the proof.
\end{proof}

\begin{proof}[Proof of Theorem~\ref{thm:primitive2}]
 For any odd integer $n \ge 3$, by Lemma~\ref{lem:abn} and noticing $ \Phi_n(\alpha,\beta) \mid P_n(\alpha,\beta)$, we have that the principal ideals $\langle \Phi_n(\alpha,\beta) \rangle$ and $\langle \alpha + \beta \rangle$ of $R$ are coprime.

Now, fix an even integer $n \ge 4$ such that $p \nmid n$ if the characteristic $p>0$.
Suppose that there exists a prime ideal $\q$ of $R$ dividing both $\Phi_n(\alpha, \beta)$ and $\Phi_2(\alpha,\beta) = \alpha + \beta$.
This means that the polynomial $x^n - y^n = \prod_{j \mid n} \Phi_j(x,y)$, viewed it as a polynomial over the quotient field $R / \q R$, has a multiple root (that is, $(\alpha, \beta)$).
However, since either $p=0$ or $p \nmid n$, we know that $x^n - y^n$ is in fact a simple polynomial.
Hence, this leads to a contradiction, and so $\langle \Phi_n(\alpha,\beta) \rangle$ and $\langle \alpha + \beta \rangle$ are coprime ideals in  $R$.

Hence, we always have that for any integer $n \ge 3$ with $p \nmid n$ if $p>0$, the ideals $\langle \Phi_n(\alpha,\beta) \rangle$ and $\langle \alpha + \beta \rangle$ are coprime in $R$.

Therefore, by construction we directly obtain from Remark~\ref{rem:Un} (choosing $\lambda = \alpha, \eta = \beta$ there) that
the primitive part of $L_n$ is exactly the ideal $\langle \Phi_n(\alpha,\beta) \rangle$, where $n$ satisfies \eqref{eq:boundn}. 

Moreover, when $n$ satisfies \eqref{eq:boundn}, 
 by Lemmas~\ref{lem:Phin} and \ref{lem:nonunit} (choosing $\lambda = \alpha, \eta = \beta$ there), we know that $\Phi_n(\alpha,\beta)$ is in $R$ and it is not a unit in $R$.

Hence, the term  $L_n$ has a primitive prime divisor when $n$ satisfies \eqref{eq:boundn}.
\end{proof}

\begin{remark}  \label{rem:Ln}
From the proof, we know that when $R$ is the integral closure of a function field, the primitive part
of $L_n$ is exactly the principal ideal $\langle \Phi_n(\alpha, \beta) \rangle$, where $n$ satisfies \eqref{eq:boundn}.
\end{remark}

\section{Proofs of Theorems~\ref{thm:strong1} and \ref{thm:primitive1}}
\label{sec:Zig}

First, Theorem~\ref{thm:strong1} follows directly from Theorem~\ref{thm:strong2}.

For the proof of Theorem~\ref{thm:primitive1}, we apply the same arguments as in \cite[Proof of Theorem 1.6]{Sha}
and use Theorem~\ref{thm:primitive2} and Remark~\ref{rem:Ln}.

\begin{proof}[Proof of Theorem~\ref{thm:primitive1}]
Fix an integer $n$ satisfying \eqref{eq:boundn}.  
Taking $\alpha = f$ and $\beta = g$ in Theorem~\ref{thm:primitive2} and noticing Remark~\ref{rem:Ln},
we know that the primitive part of the term $ (f^n - g^n) / (f-g)$ is  the ideal $\langle \Phi_n(f, g) \rangle$.
As in the proof of Theorem~\ref{thm:primitive2} and noticing $\Phi_1(f,g)=f-g$, we obtain that the ideals $\langle \Phi_n(f,g) \rangle$ and $\langle f-g \rangle$ are coprime in $R$.
Hence, the primitive part of the term $F_n = f^n - g^n$ is the ideal $\langle \Phi_n(f,g) \rangle$. 
Then, the desired result follows from Lemma~\ref{lem:nonunit} (choosing $\lambda = f, \eta = g$ there). 
\end{proof}

\begin{remark}
From the proof, we know that when $R$ is the integral closure of a function field,
the primitive part of $F_n$ is exactly
the principal ideal $\langle \Phi_n(f,g) \rangle$, where $n$ satisfies \eqref{eq:boundn}.
\end{remark}

\section*{Acknowledgement}
The authors would like to thank the referee for careful reading of their paper and valuable comments.
For the research,  Xiumei Li was supported by the National Science Foundation of China Grant No.12001312; 
and Min Sha was supported by the Key Laboratory of Applied Mathematics of Fujian Province University (Putian University) (NO. SX202201) and the Guangdong Basic and Applied Basic Research Foundation (No. 2025A1515010635).

\end{document}